\documentclass{article}
\usepackage{graphicx} % Required for inserting images

\usepackage{amsmath}
\usepackage{amsfonts}
\usepackage{amssymb}
\usepackage{tikz-cd}
\usepackage{amsthm}
\usepackage{mathrsfs}
\usepackage{changepage}
\usepackage{setspace}
\usepackage{fancyhdr}

\DeclareMathOperator{\Ind}{Ind}

\DeclareMathOperator{\N}{\mathbb{N}}

\DeclareMathOperator{\tp}{tp}

\DeclareMathOperator{\blank}{\underline{\enspace}}

\def\Ind#1#2{#1\setbox0=\hbox{$#1x$}\kern\wd0\hbox to 0pt{\hss$#1\mid$\hss}
\lower.9\ht0\hbox to 0pt{\hss$#1\smile$\hss}\kern\wd0}

\def\ind{\mathop{\mathpalette\Ind{}}}

\def\notind#1#2{#1\setbox0=\hbox{$#1x$}\kern\wd0
\hbox to 0pt{\mathchardef\nn=12854\hss$#1\nn$\kern1.4\wd0\hss}
\hbox to 0pt{\hss$#1\mid$\hss}\lower.9\ht0 \hbox to 0pt{\hss$#1\smile$\hss}\kern\wd0}

\def\nind{\mathop{\mathpalette\notind{}}}

\newtheorem{theorem}{Theorem}[section]
\newtheorem{fact}[theorem]{Fact}
\newtheorem{lemma}[theorem]{Lemma}
\newtheorem{prop}[theorem]{Proposition}
\newtheorem{cor}[theorem]{Corollary}

\theoremstyle{remark}
\newtheorem{remark}[theorem]{Remark}

\theoremstyle{definition}
\newtheorem{definition}[theorem]{Definition}

\title{Lascar Rank in Expansions}
\author{Michael Lange}
\date{August 2026}

\begin{document}

\onehalfspacing
\setlength{\headheight}{14pt}   % or 15pt if you get warnings
\setlength{\headsep}{8pt} 
\newpage
\maketitle

\section{Introduction}

Our goal is to prove the following:

\begin{theorem}
    If $T$ is a supersimple theory and $T^+$ a supersimple expansion, then $\mathrm{SU}(T)\leq \mathrm{SU}(T^+)$
\end{theorem}
\noindent Which is a corollary of:
\begin{theorem}\label{HeaderThm2}
    If $T$ is a supersimple theory, $T^+$ a supersimple expansion, and $p$ a complete $T$-type over some parameters, then $p$ can be completed to a complete $T^+$ type $p^+$ whose SU-rank in $T^+$ is at least that of $p$ in $T$.
\end{theorem}

The proof of this turns out to require some work, in contrast to the corresponding fact for, say, Morley rank which follows straightforwardly from definitions. For the Morley rank of a formula, one verifies non-decreasing in expansions by observing that the patterns which witness an increase in rank (namely, some inclusion of infinitely many disjoint definable sets inside another) are preserved in expansions. Then one verifies the claim for types by viewing complete $T$-types as partial $T^+$-types and taking a completion which does not introduce any formulas of lower rank; here it is key that Morley rank is a continuous rank on types.

The first half of this argument, preserving the patterns that witness a +1 increase in rank, is slightly more complicated in the case of SU-rank, but still brief. We will give two versions of the argument shortly. The second half does not translate in any reasonable way to SU-rank since SU-rank is not in general continuous; in other words, the SU-rank on types does not correspond to a well-behaved rank notion on formulas. This is the main problem we will have to work around in this paper.

In the case of Theorem \ref{HeaderThm2} that $p$ has finite rank, non-continuity of SU-rank poses no problem, and one need only be careful about turning instances of $T$-dividing into instances of $T^+$-dividing. One can argue inductively as follows: begin with some $T$-type $p$ of rank $\geq n+1$ over some parameters $A$. Find, in $T$, a forking extension $q$ of rank $\geq n$. Let $\varphi(x,b)$ be a formula in this extension $T$-dividing over $A$. By taking a sufficiently long $T$-indiscernible sequence in the $T$-type of $b$, we may extract an infinite $T^+$-indiscernible sequence in the $T$-type of $b$ concentrating on a single $T^+$-type, say that of some $b'$; the instances of $\varphi$ along this subsequence will still be jointly inconsistent. Hence after moving by a $T$-automorphism over $A$, we may replace $q$ by another extension of $p$ which contains a formula $\varphi(x,b')$ now $T^+$-dividing over $A$. By the inductive hypothesis, $q$ has a completion $q^+$ to a complete $T^+$-type of rank $\geq n$. Now the restriction of $q^+$ to $A$ has rank $\geq n+1$ in $T^+$, and extends the original type $p$.

Another argument for the finite rank case is presented by N\"ubling in [N]. There, the point is to choose the parameter sets in such a way that instances of $T$-dividing are automatically instances of $T^+$-dividing with no need for moving the parameters by an automorphism. The key condition can be found in [A], which says that if $A$ is a parameter set algebraically closed in the sense of $(T^+)^{eq}$ then $T^+$-independence over $A$ implies $T$-independence over $A$.

Neither argument gives an indication of how to proceed at limit stages: at stage $\omega$, the inductive hypothesis grants that $p$ has completions in $T^+$ of all finite ranks, but this is not enough to give a single completion of $p$ in $T^+$ with rank $\omega$, because SU-rank is not continuous.

But we observe that the condition on parameter sets in the latter argument is, roughly speaking, type-definable in $T^+$. Then, after committing to specific dividing formulas and dividing numbers for a tree (with, say, foundation rank $\alpha$) of $T$-forking extensions of a $T$-type $p$, we obtain a type-definable condition whose satisfiability would witness the rank of some completion of $p$ being $\geq \alpha$ in the expansion. Our idea is to use compactness to prove satisfiability of these conditions at limit stages. For this, it is necessary that the corresponding tree of formulas and dividing numbers be sufficiently uniform; more precisely, we will require that certain portions of this tree pattern actually arise from \textit{chains} of forking extensions in the reverse order of $\omega$. So we reduce to a problem entirely within $T$, of finding sufficiently uniform patterns of forking extensions for a given type.

In the case of types with countable SU-rank (so in particular, whenever the theory $T$ itself is countable), the uniformity we find has a particularly simple statement:

\begin{theorem}\label{HeaderThm3}
    Let $p$ be a complete type of countable rank $\alpha$ in a supersimple theory. Then $p$ has a chain of forking extensions in the reverse order type of $\alpha$.
\end{theorem}

This is proved by generalizing a result of Buechler [B] which states that two points which are immediately successive in the fundamental order of a superstable theory must have immediately successive U-ranks. The argument we use here does not seem to work in the case of uncountable ranks, so for these we consider trees proper, and it turns out that the amount of uniformity needed to prove Theorem \ref{HeaderThm2} is considerably less than what is obtained in Theorem \ref{HeaderThm3}. We will treat the countable case first, as the notation can be kept much simpler while the main idea of the general case is still present.

\section{Chain Patterns}
To start, we work in $T$ a simple theory.
\begin{definition}
    A \textit{chain pattern} of rank $\alpha$ is a sequence of pairs $\{(\varphi_\beta(x,y_
    \beta),k_\beta)\}_{\beta<\alpha}$ where the $\varphi_\beta(x,y_\beta)$ are bipartitioned formulas all having the same object tuple $x$ and the $k_\beta$ are positive integers. If $C$ is a chain pattern of rank $\alpha$, we will say that $C$ \textit{is a chain pattern for/over} a complete type $p$ if $p$ has a chain of forking extensions $\{q_\beta\}_{\beta<\alpha}$ such that for all $\gamma<\beta<\alpha$, $q_\gamma$ is a forking extension of $q_\beta$ as witnessed by $k_\gamma$-dividing of an instance of $\varphi_\gamma$. If there is a type of maximum index $\beta$ in the chain (i.e. if $\alpha=\beta+1$) then we also require that the forking of $q_\beta$ over $p$ is witnessed by $k_\beta$-dividing of an instance of $\varphi_\beta$. When such a chain of extensions exists, we may also say that $p$ \textit{has a forking chain of type} $C$.
    
    We should emphasize that chain patterns are well-ordered by $\supseteq$ rather than $\subseteq$; in particular, existence of infinite chain patterns is not contradicted by supersimplicity.
\end{definition}
If $C,D$ are chain patterns of rank $\alpha,\beta$ respectively, then we may form their concatenation $C^\frown D$ which is a chain pattern of rank $\alpha+\beta$.

If a chain pattern $C$ is an initial segment of a chain pattern $D$ then we write $C\trianglelefteq D$.

If $C$ is a chain pattern of rank $\alpha$ then for any $\beta<\alpha$ we define its restriction $C_\beta$ by restricting the domain of the sequence to $\beta$.

The first useful property about chain patterns is that they are apt for compactness arguments:
\begin{prop}\label{LimitChains}
    Let $C$ be a chain pattern of limit rank $\alpha$. Then a type $p\in S(A)$ has a forking chain of type $C$ iff it has a forking chain of type $C_\beta$ for all $\beta<\alpha$.
\end{prop}
\begin{proof}
    The forward implication is obvious. For the reverse implication, we first observe that existence of a forking chain of type $C$ for $p$ is equivalent to satisfiability of a certain partial type defined over $A$. Namely, we define $\pi_C(x, (y_\beta)_{\beta<\alpha})$ as follows: $\pi_C$ specifies that $x$ satisfies $p$, and that for all $\beta<\alpha$, the formula $\varphi_\beta(x,y_\beta)$ holds and the formula $\varphi_\beta(x,y_\beta)$ $k_\beta$-divides over the set $A\cup\{y_\gamma\}_{\gamma>\beta}$. Now it is clear that if $(a,(b_\beta)_{\beta<\alpha})$ are a realization of this partial type, then taking $q_\beta:=\tp(a/A(b_\gamma)_{\gamma\geq \beta})$ gives a forking chain of type $C$ over $p$. Conversely if $(q_\beta)_{\beta<\alpha}$ is a forking chain of type $C$ over $p$, then if we take $b_\beta$ to be the parameters from $q_\beta$ such that $q_\beta$ contains $\varphi_\beta(x,b_\beta)$ which $k_\beta$-divides over $q_{\beta+1}$, and $a$ to be any realization of $q_0$, then $(a,(b_\beta)_{\beta<\alpha})$ realize $\pi_C$.

    Similarly, we may see that for any $\beta<\alpha$, the restriction of $\pi_C$ to the variables $(x,(y_\gamma)_{\gamma<\beta})$ is a partial type satisfiable iff $p$ has a forking chain of type $C_\beta$; we may call this restricted partial type $\pi_{C_\beta}$. Now since $\alpha$ is a limit ordinal, every finite part of $\pi_C$ is contained in some $\pi_{C_\beta}$, and all of these are satisfiable by hypothesis, so $\pi_C$ is satisfiable by compactness.
\end{proof}

We recall the following standard fact about dividing:
\begin{fact}\label{Standard}
    Let $A$ be a set of parameters and $a,b$ (possibly infinite) tuples. Then the following are equivalent:
    \begin{enumerate}
        \item $\tp(a/Ab)$ does not divide over $A$
        \item Given an $A$-indiscernible sequence $I$ starting with $b$, there exists $a'\equiv_{Ab}a$ such that $I$ is $Aa'$-indiscernible
        \item Given an $A$-indiscernible sequence $I$ starting with $b$, there exists $J\equiv_{Ab} I$ such that $J$ is $Aa$-indiscernible.
    \end{enumerate}
\end{fact}
From this we also have:
\begin{fact}\label{DivFact}
    Let $\varphi(x,b)$ be a formula with parameters, let $A$ be a set of parameters, and let $a$ be a (possibly infinite) tuple such that $\tp(a/Ab)$ does not divide over $A$. Then if $\varphi(x,b)$ $k$-divides over $A$, it also $k$-divides over $Aa$.
\end{fact}
\begin{proof}
    Let $I$ be an $A$-indiscernible sequence starting with $b$ such that the corresponding instances of $\varphi$ are $k$-inconsistent. Then by Fact \ref{Standard}, there is $J\equiv_{Ab} I$ such that $J$ is $Aa$-indiscernible. Then $J$ still begins with $b$, and the corresponding instances of $\varphi$ are still $k$-inconsistent, so we see that $\varphi(x,b)$ $k$-divides over $Aa$.
\end{proof}

Now we give a slight refinement of the usual proof that nonforking extensions do not drop Lascar rank:
\begin{prop}\label{ChainExtProp}
    If $p\in S(A)$ has a forking chain of type $C$ and $q\in S(B)$ is a nonforking extension of $p$, then $q$ also has a forking chain of type $C$.
\end{prop}
\begin{proof}
    We induct on the rank of $C$. If the rank of $C$ is 0 then $C$ is empty and there is nothing to prove.

    If the rank of $C$ is $\alpha+1$, then write $C=C_\alpha^\frown\{(\varphi_\alpha,k_\alpha)\}$. Take some forking chain of type $C$ over $p$, and let $p'\in S(D)$ be the $\alpha$th type in the chain; i.e. the extension of $p$ which forks over $A$ by $k_\alpha$-dividing of an instance of $\varphi_\alpha$. By taking an $A$-automorphism of a large saturated model moving a realization of $p'$ to a realization of $q$, we may replace $p'$ by a type jointly consistent with $q$ (note that $p'$ will still have its own chain of type $C_\alpha$ and will still contain an instance of $\varphi_\alpha$ which $k_\alpha$-divides over $A$). Let $a\models p'\cup q$. There exists $D'\equiv_{Aa} D$ such that $D'\ind_{Aa} B$. Then after replacing $p'$ by $\tp(a/D')$ we may further assume that $D\ind_{Aa}B$. Since we have $a\ind_A B$, we may conclude by transitivity that $Da\ind_A B$. In particular, by monotonicity, base monotonicity, and normality, we have $D\ind_A B$ and $a\ind_D BD$. Hence if we let $q':=\tp(a/BD)$, then $q'$ is a nonforking extension of $p'$. Since $p'$ has a forking chain of type $C_\alpha$, it follows by induction that $q'$ has one as well. Moreover, we have that $p'$ contains an instance $\varphi_\alpha(x,d)$ with $d\in D$, which $k_\alpha$-divides over $A$. If $\overline{b}$ is a sequence enumerating $B$ then by the earlier $D\ind_A B$ and symmetry, we have that $\tp(\overline{b}/D)$ does not divide over $A$, and in particular $\tp(\overline{b}/Ad)$ does not divide over $A$. Then by Fact \ref{DivFact}, $\varphi_\alpha(x,d)$ $k_\alpha$-divides over $B$, and this formula is also in $q'$. Then taking $q'$ together with its own forking chain of type $C_\alpha$, we find a forking chain of type $C$ for $q$.

    If the rank of $C$ is a limit $\alpha$, then $p$ has a forking chain of type $C_\beta$ for all $\beta<\alpha$ and by induction so does $q$. Then by Proposition \ref{LimitChains}, $q$ has a forking chain of type $C$.
\end{proof}

\section{Finding Intermediate Extensions}
For this section, we work in $T^{eq}$ for a supersimple theory $T$. Given two ordinals $\beta,\gamma$ and a third ordinal $\alpha$, let us write $\beta\approx_\alpha \gamma$ to mean that the Cantor normal forms of $\beta,\gamma$ have the same terms for all exponents $\geq \alpha$. Observe that $\approx_\alpha$ is an equivalence relation on ordinals. Also, if $\gamma\leq \delta\leq \beta$ and $\gamma\approx_\alpha\beta$ then $\gamma\approx_\alpha\delta\approx_\alpha\beta$.

Let us also recall the following quantitative refinement of forking symmetry (see, e.g. [W] for details):
\begin{fact}
    Let $T$ be a simple theory. If $\mathrm{SU}(a/Ab)<\infty$ and $\mathrm{SU}(a/A)\geq \mathrm{SU}(a/Ab)+\omega^\alpha\cdot k$ then also $\mathrm{SU}(b/A)\geq \mathrm{SU}(b/Aa)+\omega^\alpha\cdot k$.
\end{fact}

Now we may prove a refinement of a result from [B]:

\begin{lemma}\label{EvenBetterBuechler}
    Let $p\in S(A)$ be a complete type and $q\in S(N)$ an extension of $p$, with $N$ being $|A|^+$-saturated. Suppose also that $\mathrm{SU}(q)+\omega^\alpha\leq\mathrm{SU}(p)$. Then there is an intermediate extension $p\subseteq r\subseteq q$ with $\mathrm{SU}(r)\approx_\alpha \mathrm{SU}(q)+\omega^\alpha$. %In particular $r\in S(Ab)$ for some finite tuple $b\in N$.
\end{lemma}

\begin{proof}
    By elimination of hyperimaginaries, we have $\mathrm{Cb}(q)\subseteq N$ and by supersimplicity, some finite tuple $c\in \mathrm{Cb}(q)$ such that $q$ does not fork over $c$. If $a$ is any realization of $q$, we have $\mathrm{SU}(a/Ac)+\omega^\alpha\leq\mathrm{SU}(a/A)$. Then by quantitative forking symmetry we have $\mathrm{SU}(c/A)\geq \mathrm{SU}(c/Aa)+\omega^\alpha$. In particular, we have that $\mathrm{SU}(c/A)\geq \omega^\alpha$ and so there exists a finite tuple $b$ such that $\mathrm{SU}(c/Ab)=\omega^\alpha$. By saturation, it may be assumed that $b\in N$. Let $a\models q$ so in particular $a\ind_{Ac}b$.

We claim that $a\nind_{Ab} c$. Indeed, suppose not: then by transitivity we have that $a\ind_{Ab} N$ and so $\mathrm{Cb}(a/N)\subseteq \mathrm{acl}(Ab)$; in particular, $c\in \mathrm{acl}(Ab)$ which contradicts $\mathrm{SU}(c/Ab)=\omega^\alpha$.

Now by symmetry, $c\nind_{Ab} a$ which in particular means $\mathrm{SU}(c/Aab)<\omega^\alpha$.

Let us consider two ways of bounding $\mathrm{SU}(ac/Ab)$ using the Lascar inequalities. First:

\[\mathrm{SU}(c/Aab)+\mathrm{SU}(a/Ab)\leq\mathrm{SU}(ac/Ab)\leq\mathrm{SU}(c/Aab)\oplus\mathrm{SU}(a/Ab)\]

This inequality and the fact that $\mathrm{SU}(c/Aab)<\omega^\alpha$ yield:

\[\mathrm{SU}(a/Ab)\approx_\alpha\mathrm{SU}(ac/Ab)\]

Now the other:

\[\mathrm{SU}(a/Abc)+\mathrm{SU}(c/Ab)\leq\mathrm{SU}(ac/Ab)\leq\mathrm{SU}(a/Abc)\oplus\mathrm{SU}(c/Ab)\]

The assumption that $\mathrm{SU}(c/Ab)=\omega^\alpha$ implies that 
\[\mathrm{SU}(a/Abc)\oplus\mathrm{SU}(c/Ab)\approx_\alpha\mathrm{SU}(a/Abc)+\mathrm{SU}(c/Ab)=\mathrm{SU}(q)+\omega^\alpha\]

and so the previous inequality yields:
\[\mathrm{SU}(ac/Ab)\approx_\alpha\mathrm{SU}(q)+\omega^\alpha\]

which together with our first $\approx_\alpha$ equivalence gives:

\[\mathrm{SU}(a/Ab)\approx_\alpha\mathrm{SU}(q)+\omega^\alpha\]

Hence we may take $r=\tp(a/Ab)$ to conclude the claim.
\end{proof}

Using this lemma, we work towards building chain patterns over types which are as long as the Lascar rank will allow:

\begin{prop}\label{IntermediateProp}
    Suppose $p$ has an extension $q$ with a chain pattern $C$ of rank $\beta$. Suppose also that $\mathrm{SU}(q)+\omega^\alpha\leq \mathrm{SU}(p)$ where $\alpha<\omega_1$. Then $p$ also has some chain pattern $D\trianglerighteq C$ of rank $\beta+\omega^\alpha$.
\end{prop}
\begin{proof}
    We will prove this by induction on $\alpha$. In the case $\alpha=0$ we have by assumption that $q$ forks over $p$. This is witnessed by some dividing data $(\varphi_\beta, k_\beta)$ and we see that $C^\frown \{(\varphi_\beta, k_\beta)\}$ is then a chain pattern for $p$ of rank $\beta+1$.

    Now take the case $\alpha+1$. Say $p\in S(A)$. By assumption $\mathrm{SU}(q)+\omega^{\alpha+1}\leq\mathrm{SU}(p)$, and so also $\mathrm{SU}(q)+\omega^\alpha<\mathrm{SU}(p)$. By Proposition \ref{ChainExtProp}, we are free to replace $q$ by a nonforking extension and so assume $q\in S(N)$ where $N$ is $|A|^+$-saturated. By Lemma \ref{EvenBetterBuechler} we find a type $r_1$ with $p\subseteq r_1\subseteq q$ and $\mathrm{SU}(r_1)\approx_\alpha\mathrm{SU}(q)+\omega^\alpha$. In particular, $\mathrm{SU}(q)+\omega^\alpha\leq \mathrm{SU}(r_1)$ and so by the inductive hypothesis, we may find a chain pattern $C_1\trianglerighteq C$ over $r_1$ (hence over $p$) where the rank of $C_1$ is $\beta+\omega^\alpha$. Suppose we have defined chain patterns $C_k\trianglerighteq\dots\trianglerighteq C_1\trianglerighteq C$ over $p$ where each $C_i$ has rank $\beta+\omega^\alpha\cdot i$, and $C_k$ is also a chain pattern over some $r_k$ an extension of $p$ satisfying $\mathrm{SU}(r_k)\approx_\alpha\mathrm{SU}(q)+\omega^\alpha\cdot k$. Then after replacing $r_k$ by a nonforking extension we may find an intermediate type $p\subseteq r_{k+1}\subseteq r_k$ with $\mathrm{SU}(r_{k+1})\approx_\alpha\mathrm{SU}(r_k)+\omega^\alpha=\mathrm{SU}(q)+\omega^\alpha\cdot(k+1)$. In particular, $\mathrm{SU}(r_k)+\omega^\alpha\leq \mathrm{SU}(r_{k+1})$ and so by the inductive hypothesis we find $C_{k+1}\trianglerighteq C_{k}$ a chain pattern over $r_{k+1}$ with rank $\beta+\omega^\alpha\cdot(k+1)$.

    This completes the definition of an ascending $\omega$-chain of chain patterns over $p$, $C\trianglelefteq C_{1}\trianglelefteq C_2\trianglelefteq\dots$. The union of these chain patterns is a chain pattern $D$ of rank $\beta+\omega^{\alpha+1}$ and it follows by Proposition \ref{LimitChains} that $p$ still has a forking chain of type $D$. 

    Now take the case $\alpha$ a limit. Since $\alpha$ is countable, we can find $\gamma_1,\gamma_2,\dots$ a strictly increasing sequence of ordinals $<\alpha$ whose supremum is $\alpha$. Then $\mathrm{SU}(q)+\omega^\alpha\leq \mathrm{SU}(p)$ implies $\mathrm{SU}(q)+\omega^{\gamma_1}\leq\mathrm{SU}(p)$. As before we may assume $q$ is over a saturated model and so obtain $r_1$ with $p\subseteq r_1\subseteq q$ and $\mathrm{SU}(r_1)\approx_{\gamma_1}\mathrm{SU}(q)+\omega^{\gamma_1}$. By inductive hypothesis, we find $C_{1}\trianglerighteq C$ a chain pattern over $r_1$ of rank $\beta+\omega^{\gamma_1}$. 
    
    Suppose we have defined chain patterns $C_k\trianglerighteq\dots\trianglerighteq C_1\trianglerighteq C
    $ over $p$ where each $C_i$ has rank $\beta+\omega^{\gamma_i}$, and $C_k$ is also a chain pattern over some type $r_k$ an extension of $p$ satisfying $\mathrm{SU}(r_k)\approx_{\gamma_k}\mathrm{SU}(q)+\omega^{\gamma_k}$. Then we have $\mathrm{SU}(r_k)+\omega^{\gamma_{k+1}}= \mathrm{SU}(q)+\omega^{\gamma_{k+1}}\leq \mathrm{SU}(p)$. After replacing $r_k$ by an extension over a saturated model, we find $p\subseteq r_{k+1}\subseteq r_{k}$ with $\mathrm{SU}(r_{k+1})\approx_{\gamma_{k+1}}\mathrm{SU}(r_k)+\omega^{\gamma_{k+1}}=\mathrm{SU}(q)+\omega^{\gamma_{k+1}}$. By the inductive hypothesis, we may find a chain pattern $C_{k+1}$ of rank $\beta+\omega^{\gamma_k}+\omega^{\gamma_{k+1}}=\beta+\omega^{\gamma_{k+1}}$ extending $C_k$ over $r_{k+1}$. 
    
    This completes the inductive definition of an ascending chain of chain patterns $C\trianglelefteq C_1\trianglelefteq C_2\trianglelefteq\dots$ over $p$ with $C_i$ having rank $\beta+\omega^{\gamma_i}$. Then the union of this chain is a chain pattern $D\trianglerighteq C$ of rank $\beta+\omega^\alpha$, and again by compactness $D$ is still a chain pattern over $p$.
    \end{proof}

This will yield our first main result:
\begin{theorem}\label{ChainThm}
    Let $\alpha<\omega_1$ be an ordinal and $p\in S(A)$ be a type in a supersimple theory with $\mathrm{SU}(p)\geq\alpha$. Then there exists a chain pattern over $p$ of rank $\alpha$.
\end{theorem}
\begin{proof}
    Write $\alpha$ in Cantor normal form $\alpha=\omega^{\gamma_1}\cdot k_1+\dots+\omega^{\gamma_n}\cdot k_n$. We induct on $n$.

    If $\alpha=\omega^{\gamma}\cdot k$ then again we proceed by induction on $k$. If $\alpha=\omega^{\gamma}$ then let $q$ be some algebraic extension of $p$. Trivially, $q$ has the empty chain pattern, and $\mathrm{SU}(q)+\omega^\gamma\leq \mathrm{SU}(p)$ so Proposition \ref{IntermediateProp} gives a chain pattern of rank $\omega^\gamma$ over $p$. Now if $\alpha=\omega^\gamma\cdot (k+1)$, then $p$ has an extension of rank $\omega^\gamma\cdot k$ which has a chain pattern of rank $\omega^\gamma\cdot k$ by induction, so we apply the proposition again to get a chain pattern of rank $\omega^\gamma\cdot(k+1)$ over $p$.

    Now suppose $\alpha$ has normal form $\alpha=\omega^{\gamma_1}\cdot k_1+\dots+\omega^{\gamma_n}\cdot k_n+\omega^{\gamma_{n+1}}\cdot k_{n+1}$. Let $\beta:=\omega^{\gamma_1}\cdot k_1+\dots+\omega^{\gamma_n}\cdot k_n$. Then $p$ has an extension $q$ of rank $\beta$ which by hypothesis has a chain pattern of rank $\beta$. Then just as in the previous case, we induct on $k_{n+1}$ to conclude the theorem.
\end{proof}

\section{Lascar rank in expansions}\label{ChainConclusionSection}
Now we additionally fix some supersimple expansion $T^+$ of $T$. We write $\mathrm{acl}^+$ for algebraic closure in the sense of $(T^+)^{eq}$; similarly $\mathrm{SU}^+, \ind^+, \tp^+$ all indicate the corresponding notions relative to $(T^+)^{eq}$ while undecorated symbols are relative to $T^{eq}$. We may always work in a monster model $(\mathbb{M}^+)^{eq}$ of $(T^+)^{eq}$ whose reduct $\mathbb{M}^{eq}$ to $\mathcal{L}^{eq}$ is itself a monster model of $T^{eq}$.

A key fact for supersimple expansions, referenced in the introduction, is the following: if $A,C\subseteq \mathbb{M}^{eq}$ and $B=\mathrm{acl}^+(B)\subseteq (\mathbb{M}^+)^{eq}$ then $A\ind^+_BC$ implies $A\ind_{B\cap \mathbb{M}^{eq}} C$ (this can be found as Exercise 3.5 of [A]). We will want to use the contrapositive: say $A,B,C\subseteq \mathbb{M}^{eq}$ and $B$ is relatively $\mathrm{acl}^+$-closed. Then $A\nind_BC$ implies $A\nind_{\mathrm{acl}^+(B)}^+C$. 

Note that $(T^+)^{eq}$ possibly contains sorts not in $T^{eq}$ so we only ever say that a subset $B\subseteq \mathbb{M}^{eq}$ is \textit{relatively} $\mathrm{acl}^+$-closed; i.e. $B=\mathrm{acl}^+(B)\cap\mathbb{M}^{eq}$.

\begin{theorem}\label{CountableThm}
    Let $\alpha<\omega_1$ and let $p\in S_T(A)$ be a complete $T$-type of rank $\mathrm{SU}(p)\geq\alpha$. Then there exists a complete $T^+$-type $p^+\in S_{T^+}(A)$ extending $p$ with rank $\mathrm{SU}^+(p^+)\geq \alpha$.
\end{theorem}
\begin{proof}
    First replace $p$ by a nonforking extension to $\mathrm{acl}^+(A)\cap \mathbb{M}^{eq}$ (at the end of the argument, we may replace $p^+$ by its restriction to $A$, by which the rank can only possibly go up). By Theorem \ref{ChainThm} we have a chain pattern $C$ over $p$ of rank $\alpha$. For every $\beta<\alpha$, for every sort $E$ of $T^{eq}$, expand $(\mathcal{L}^+)^{eq}$ by predicates $P_\beta^E$ in the respective sorts. Add also constants for $A$ and a single constant $a$. Define an expansion $T^+_C$ of $T^+$ in this language as follows: specify that the $T$-type of $a$ over $A$ is precisely $p$ and that $P_\beta^E\supseteq P_\gamma^E$ whenever $\beta\leq\gamma$. For each $\beta$, specify that the predicates $P^E_\beta$ together form a relatively $\mathrm{acl}^+$-closed set (this will involve a separate axiom for each possible size of a finite solution set for each formula of $(\mathcal{L}^+)^{eq}$ whose free variables are all of sorts available in $T^{eq}$). We will informally refer to the union of the $P_\beta^E$ as $P_\beta$. Next, for each $\beta$, specify that $a$ satisfies some instance $\varphi_\beta(x,y_\beta)$ where the $y_\beta$ are taken from $P_\beta$ and there exist an infinite sequence of elements which are indiscernible over $P_{\beta+1}$ in the sense of $T^{eq}$, satisfying the $T^{eq}$-type of $y_\beta$ over $P_{\beta+1}$, such that the corresponding sequence of instances of $\varphi_\beta(x,y_\beta)$ is $k_\beta$-inconsistent (that is: express the satisfiability of the appropriate partial type by a list of existential formulas). If $\alpha$ is a successor then by convention let us say $P_\alpha=A$.

    If this theory is satisfiable, this means we can find in some model of $T^+$, hence in $(\mathbb{M}^+)^{eq}$, a tuple $a\in\mathbb{M}^{eq}$ and relatively $\mathrm{acl}^+$-closed parameter sets $A_\beta\subseteq \mathbb{M}^{eq}$ (the interpretations of the $P_\beta$) such that if $q_\beta:=\tp(a/A_\beta)$ then $q_\beta$ always contains an instance of $\varphi_\beta$ which $k_\beta$-divides over $A_{\beta+1}$ in the sense of $T$, and $q_0\supseteq q_1\supseteq\dots \supseteq p$. By Adler's condition, this is enough to show that $a\nind^+_{\mathrm{acl}^+(A_{\beta+1})} A_\beta$. %[Come back later to show the actual forking formulas are preserved].
    Hence we have a chain of $T^+$-forking extensions of rank $\alpha$ over $p^+:=\tp^+(a/\mathrm{acl}^+(A))$.

    So we need to show that this theory is finitely satisfiable. All remaining mentions of forking/dividing are in the sense of $T$. To show that the above theory is finitely satisfiable, it suffices to realize any finite part of $C$ by forking extensions which are over relatively $\mathrm{acl}^+$-closed parameter sets, which we may do by induction on the number of extensions $n$:

    In the case $n=1$ we must find some extension $q$ of $p$ over a relatively $\mathrm{acl}^+$-closed set which forks over $A$ as witnessed by $k_\beta$-dividing of some instance of $\varphi_\beta$. We already know that $p$ has some extension $q\in S_T(B),\; A\subseteq B\subseteq \mathbb{M}^{eq}$ with the appropriate dividing witness. Now take a nonforking extension of $q$ to $\mathrm{acl}^+(B)\cap \mathbb{M}^{eq}$; the forking of $q$ over $p$ is still witnessed by the same data.

    In the case $n+1$ we need to find a forking chain of type $\{(\varphi_{\beta_1},k_{\beta_1}),\dots,(\varphi_{\beta_{n+1}},k_{\beta_{n+1}})\}$ where this finite chain pattern is a subsequence of $C$. To ease notation let us simply write $(\varphi_{\beta_i},k_{\beta_i})=:(\varphi_i,k_i)$. By the fact that $C$ is a chain pattern over $p$, we find an extension $q\in S_T(B)$ which forks over $p$ as witnessed by $(\varphi_{n+1},k_{n+1})$ where $q$ itself has a forking chain of type $\{(\varphi_1,k_1),\dots,(\varphi_n,k_n)\}$. Replace $q$ by a nonforking extension to $\mathrm{acl}^+(B)\cap\mathbb{M}^{eq}$. Then by Proposition \ref{ChainExtProp}, $q$ still has a forking chain of type $\{(\varphi_1,k_1),\dots,(\varphi_n,k_n)\}$, and by the inductive hypothesis we may assume all types in this chain have relatively $\mathrm{acl}^+$-closed parameter sets. Also, $(\varphi_{n+1},k_{n+1})$ still witness the forking of $q$ over $p$ so we are done.
\end{proof}

Since the types in countable supersimple theories have at most countable rank, this immediately yields:
\begin{cor}
    Let $T$ be a countable supersimple theory and $T^+$ a supersimple expansion. Then $\mathrm{SU}(T)\leq\mathrm{SU}(T^+)$, where by the SU-rank of a theory, we mean the supremum of SU-ranks of 1-types in this theory.
\end{cor}

\section{Forking Trees}\label{Trees}
We continue with the setting and notation established at the beginning of Section \ref{ChainConclusionSection}. To remove the cardinality restriction on $\alpha$ in Theorem \ref{CountableThm}, we will have to more carefully consider the data witnessing the SU-rank of a complete type. To organize these, we will consider certain trees of types and formulas.

\begin{definition}
    The \textit{standard tree of rank} $\alpha$ is a certain partial order (denoted $\trianglelefteq$) whose isomorphism type we describe inductively. We will use $\mathcal{T}_\alpha$ to denote the standard tree of rank $\alpha$:

    \begin{itemize}
        \item The standard tree of rank 0 consists of a single point.

        \item The standard tree of rank $\alpha+1$ consists of a copy of $\mathcal{T}_\alpha$ together with a single new node (the root) which succeeds every node of $\mathcal{T}_\alpha$.

        \item For $\alpha$ a limit, a tree of rank $\alpha$ consists of a disjoint union of copies of $\mathcal{T}_\beta$ for all $\beta<\alpha$ together with a new node (the root) which succeeds every node of each of these copies.
    \end{itemize}
\end{definition}
We will freely use familiar language associated with trees: the root is the maximal element for $\trianglelefteq$; a descendant is a predecessor, and a child is an immediate predecessor.

\begin{definition}
    A forking tree of rank $\alpha$ is a function $\tau$ from $\mathcal{T}_\alpha$ to some disjoint union of type spaces over small parameter sets such that if $q$ is a child of $p$, then $q$ is a forking extension of $p$.

    We will always denote the free tuple of the types in such forking trees by $x$.
\end{definition}

Strictly speaking, the above should say that if $s,t$ are nodes of $\mathcal{T}_\alpha$ with $t$ a child of $s$, then $\tau(t)$ is a forking extension of $\tau(s)$. But using the terms ``root", ``child", etc. for the types in the image of $\tau$ will never cause confusion, so we do so freely.

\noindent If $p$ is the root of a given forking tree $\tau$, then we say $\tau$ is a forking tree for $p$.
\\\\
\noindent If $t$ is a node of $\mathcal{T}_\alpha$ then the set of non-strict descendants $\mathcal{S}(t)$ of $t$ is a copy of $\mathcal{T}_\beta$ for some $\beta\leq \alpha$ (indeed, for $\beta$ equal to the foundation rank of $t$ in $\trianglelefteq$, which in our context is always well-founded). We will call $\beta$ the rank of the node $t$. If $\tau$ is a forking tree of rank $\alpha$ and $t$ has rank $\beta$ then $\tau\mid_{\mathcal{S}(t)}$ is a forking tree of rank $\beta$. We will call this the tree above $\tau(t)$; note that the tree above $\tau(t)$ includes $\tau(t)$.

\begin{definition}
    Given a forking tree $\tau$ with domain $\mathcal{T}_\alpha$, we define a \textit{tree of formulas} or a \textit{shape} $\sigma$ for $\tau$ as follows: letting $r$ denote the root node of $\mathcal{T}_\alpha$, and $\mathcal{L}(x,\blank)$ the set of bipartitioned $\mathcal{L}$-formulas whose first partition is $x$, then $\sigma$ is a function $\mathcal{T}_\alpha\to \mathcal{L}(x, \blank)\times\N$ such that for any non-root node $t\in\mathcal{T}_\alpha$, if $\sigma(t)=(\varphi(x,y),k)$ then the forking of $\tau(t)$ over its parent can be witnessed by $k$-dividing of some instance of $\varphi(x,y)$.

    We will also say that $\alpha$ is the rank of $\sigma$.
\end{definition}
We observe here that the condition for a $\sigma$ to be a shape for $\tau$ is vacuous on the root of $\sigma$; nevertheless, when considering subtrees later, it will be more convenient to keep the root as part of the data of a tree of formulas.

\begin{remark}
    The following are clear:
    \begin{itemize}
        \item Every type of Lascar rank $\geq \alpha$ has a forking tree of rank $\alpha$.
        \item Every forking tree has a shape of equal rank.
    \end{itemize}
\end{remark}

\begin{definition}
    When $\alpha$ is a limit ordinal, for any $\beta<\alpha$ we define the \textit{distinguished subtree of }$\mathcal{T}_\alpha$ of rank $\beta$, which we will denote $\mathcal{T}_\alpha[\beta]$ as follows: If $\beta=\gamma+1$ is a successor ordinal, we take $\mathcal{T}_\alpha[\beta]$ to be the root of $\mathcal{T}_\alpha$ together with the tree above the unique root-child of rank $\gamma$. If $\beta$ is a limit, then we take $\mathcal{T}_\alpha[\beta]$ to be the union of all $\mathcal{T}_\alpha[\gamma]$ for $\gamma<\beta$.
\end{definition}

Similarly, if $\tau$ is a tree (of types or of formulas) of rank $\alpha$ a limit then we take the distinguished subtree of $\tau$ of rank $\beta<\alpha$ to be $\tau_\beta:=\tau\mid \mathcal{T}_\alpha[\beta]$.

When $\gamma<\beta<\alpha$, with $\beta,\alpha$ both limits, it is clear from the definition that $\mathcal{T}_\alpha[\beta][\gamma]=\mathcal{T}_\alpha[\gamma]$, and hence that $(\tau_\beta)_\gamma=\tau_\gamma$.

The following is the analogue of Proposition \ref{ChainExtProp}, and we have already proven the nontrivial case: 
\begin{prop}\label{TreeExtProp}
    If $p\in S(A)$ has a forking tree $\tau$ of shape $\sigma$ and $q\in S(B)$ is a nonforking extension of $p$, then $q$ also has a forking tree of shape $\sigma$.
\end{prop}
\begin{proof}
    The proof is exactly the same as for Proposition \ref{ChainExtProp} at successor stages. The limit stage is now essentially trivial (requiring no compactness argument), once we note: for a shape $\sigma$ of limit rank $\alpha$, a type has a a forking tree of shape $\sigma$ iff it has a forking tree of shape $\sigma_\beta$ for all $\beta<\alpha$. Indeed, it is enough to have such a forking tree for all $\beta<\alpha$ which are successor ordinals.
\end{proof}

\begin{definition}
    We say that a forking tree $\tau$ is \textit{jointly consistent} if the union of all types in the image of $\tau$ is consistent. We will call a realization of this union simply ``a realization of $\tau$".
\end{definition}

We can now phrase a sufficient condition for the conclusion to Theorem \ref{HeaderThm2}:

\begin{prop}\label{GoalProp}
    Suppose $p\in S_T(A)$ with $A$ relatively $\mathrm{acl}^+$-closed has (in $T$) a forking tree of rank $\alpha$ which is jointly consistent and such that all parameter sets are relatively $\mathrm{acl}^+$-closed. Then $p$ extends to a complete type $p^+$ in $T^+$ over $\mathrm{acl}^+(A)$ such that $\text{SU}^+(p^+)\geq \alpha$.
\end{prop}

\begin{proof}
    Let $\tau$ be the hypothesized forking tree. Let $\upsilon$ be the ``tree of parameters" for $\tau$; that is, $\upsilon$ is a function with domain $\mathcal{T}_\alpha$ and for $t\in\mathcal{T}_\alpha$, $\upsilon(t)$ is the set of parameters over which $\tau(t)$ is defined. Let $a$ be a realization of $\tau$. For every non-root node $t\in\mathcal{T}_\alpha$ let $t^-$ denote the parent of $t$. Then we have, for all $t\in\mathcal{T}_\alpha$, $a\nind_{\upsilon (t^-)} \upsilon(t)$. Since $\upsilon(t^-)$ is relatively $\mathrm{acl}^+$-closed, we have that $a\nind_{\mathrm{acl}^+(\upsilon(t^-))}^+ \upsilon(t)$. Then also, $a\nind_{\mathrm{acl}^+(\upsilon(t^-))}^+ \mathrm{acl}^+(\upsilon(t))$. Hence, if we define a tree of types by $\tau^+(t):=\tp^+(a/\mathrm{acl}^+(\upsilon(t)))$ then $\tau^+$ is a forking tree (in $T^+$) of rank $\alpha$ for the type $p^+:=\tp^+(a/\mathrm{acl}^+(A))$.
\end{proof}

Hence, our goal should be figure out when a type (in $T$) which has a forking tree of rank $\alpha$ has, moreover, a jointly consistent forking tree of rank $\alpha$ with relatively $\mathrm{acl}^+$-closed parameter sets. We summarize the properties of the trees we seek:
\begin{definition}
    A forking tree $\tau$ in $T$ is called a $T^+$-robust tree of shape $\sigma$ if:
    \begin{itemize}
        \item $\tau$ is jointly consistent

        \item $\tau$ has shape $\sigma$ (so in particular both trees have the same rank $\alpha$), and moreover:

        \item Whenever $s\triangleright t$ is a parent-child pair in $\mathcal{T}_\alpha$, if $p\in S_T(A), q\in S_T(B)$ denote $\tau(s),\tau(t)$ respectively, and $\sigma(t)=(\varphi(x,y),k)$, then $B$ contains a parameter $b$ such that $\varphi(x,b)$ $k$-divides (in the sense of $T$) over $A$, $\varphi(x,b)\in q$, and $B$ contains $\mathrm{acl}^+(Ab)\cap \mathbb{M}^{eq}$.
    \end{itemize}
\end{definition}

And we observe that this definition really is sufficient:
\begin{prop}\label{PruneProp}
    If $p\in S_T(A)$ with $A$ relatively $\mathrm{acl}^+$-closed has a $T^+$-robust tree of shape $\sigma$, then it has a jointly consistent forking tree of shape $\sigma$ with relatively $\mathrm{acl}^+$-closed parameter sets.
\end{prop}
\begin{proof}
    Let $\tau$ be the tree in question, say of rank $\alpha$, and let $r$ denote the root of $\mathcal{T}_\alpha$. Let $\upsilon$ be the function on $\mathcal{T}_\alpha$ such that $\upsilon(t)$ is the parameter set of $\tau(t)$. Let $\chi$ be a function on $\mathcal{T}_\alpha\setminus\{r\}$ which picks out the parameters of dividing formulas supplied by the robustness property. That is, if $t^-$ is the predecessor of a non-root $t$ and $\sigma(t)=(\varphi(x,y),k)$ then $\chi(t)=b$ is such that $\varphi(x,b)\in \tau(t)$ with $\varphi(x,b)$ $k$-dividing over $\upsilon(t^-)$ and $\upsilon(t)$ contains $\mathrm{acl}^+(\upsilon(t^-)b)\cap\mathbb{M}^{eq}$. Then define a new tree $\tau'$ by $\tau'(r)=\tau(r)$ and for non-root $t$, $\tau'(t)=\tau(t)\upharpoonright \mathrm{acl}^+(\upsilon(t^-)\chi(t))\cap\mathbb{M}^{eq}$. Then it is easy to check that $\tau'$ is as required, and is realized by any realization of $\tau$.
\end{proof}

We wish to prove next that the existence of $T^+$-robust trees of a specified shape is type definable. For brevity of notation, a sequence of variables indexed by some linear order $I$ will regularly be abbreviated $y_I$, in place of the more correct $(y_i)_{i\in I}$. We observe several type-definable configurations:

\begin{remark}\label{DivRmk}
    For a fixed formula $\varphi(x,y)$, a fixed positive integer $k$, and a sequence of variables $z_I$, there is a partial type $\zeta_\varphi^k(y,z_I)$ such that $\models \zeta_\varphi^k(b,c_I)$ iff $\varphi(x,b)$ $k$-divides over the set enumerated by $c_I$.
\end{remark}

\begin{prop}\label{ACLProp}
    For a given sequence of variables $y_I$ (of sorts available in $T$) there exists a sequence of variables $z_J$ (of sorts available in $T$) and a partial type (in $(\mathcal{L}^{+})^{eq}$) denoted $\rho(y_I,z_J)$, such that $(\mathbb{M}^+)^{eq}\models \rho(b_I,c_J)$ implies that the set $C$ enumerated by $c_J$ contains $\mathrm{acl}^+(B)\cap\mathbb{M}^{eq}$, where $B$ is the set enumerated by $b_I$. Conversely, given $b_I$ and $B$ the set enumerated by $b_I$, the set $C:=\mathrm{acl}^+(B)\cap\mathbb{M}^{eq}$ may be enumerated as $c_J$ (possibly with repetitions) in such a way that $\models \rho(b_I,c_J)$.
\end{prop}
\begin{proof}
The claimed properties of $\rho$ will be clear from its explicit definition:

Fix an enumeration $\{\varphi_i(\overline{y}_i, z_i)\}_{i<\kappa}$ of bipartitioned $(\mathcal{L}^+)^{eq}$-formulas whose first partitions are tuples of variables from $y_I$, and $z_i$ are of sorts available in $T^{eq}$. Then the sequence $z_J$ will be the sequence $(z_i^j)_{i<\kappa, j<\omega}$ where each $z_i^j$ is of the same sort as $z_i$. The partial type $\rho$ contains, for every pair $i<\kappa$, $n<\omega$, the following formula:
\[(\exists^{=n}z_i\varphi_i(\overline{y}_i,z_i))\rightarrow(\bigwedge_{\substack{j_1,j_2<n\\ j_1\neq j_2}}z_i^{j_1}\neq z_i^{j_2}\land \bigwedge_{j<n}\varphi_i(\overline{y}_i, z_i^j))\]
    
\end{proof}

\begin{prop}\label{TypeDefProp}
    For any tree of $\mathcal{L}^{eq}$-formulas $\sigma$ and any sequence of variables $y_I$, there is a partial type (in $(\mathcal{L}^+)^{eq}$), which we call $\pi_\sigma(x, y_I)$, such that $(\mathbb{M}^+)^{eq}\models \pi_\sigma(a, b_I)$ iff the complete $\mathcal{L}^{eq}$-type $\tp(a/B)$ has a $T^+$-robust tree of shape $\sigma$ realized by $a$, where $B$ denotes the set enumerated by $b_I$.
\end{prop}
\begin{proof}
    Induct on $\alpha$, the rank of $\sigma$. Note: all mentions of dividing in the proof will be in the sense of $T$.\\
    
    \noindent Case $\alpha=0$: here $\pi_\sigma(x,y_I)$ may be taken as the trivial type.\\
    \\
    \noindent Case $\alpha+1$: Let $r$ denote the root node of $\mathcal{T}_{\alpha+1}$ and $r^+$ its unique child. Let $\sigma_0$ denote $\sigma$ less its root node. Let $(\varphi(x,w),k)$ be $\sigma(r^+)$. By Remark \ref{DivRmk}, there is a partial $\mathcal{L}$-type $\zeta(w, y_I)$ such that $\mathbb{M}^{eq}\models\zeta(c,b_I)$ iff $\varphi(x,c)$ $k$-divides over the set enumerated by $b_I$. By Proposition \ref{ACLProp} there is also a sequence of variables $z_J$ and a partial $\mathcal{L}^+$-type $\rho(y_Iw,z_J)$ such that $(\mathbb{M}^+)^{eq}\models \rho(b_Ic,c_J)$ implies that the set enumerated by $c_J$ contains $\mathrm{acl}^+(b_Ic)\cap\mathbb{M}^{eq}$ and conversely, if $C\supseteq \mathrm{acl}^+(b_Ic)\cap\mathbb{M}^{eq}$ then it is possible to enumerate some $C'\subseteq C$ as $c_J$ in such a way that $(\mathbb{M}^+)^{eq}\models \rho(b_Ic,c_J)$. Let $\pi_{\sigma_0}(x, z_J)$ be supplied by inductive hypothesis. Then we define 
    \[\chi_{\sigma}(x,y_I,w,z_J):=\{\varphi(x,w)\}\cup \zeta(w, y_I)\cup \rho(y_Iw,z_J)\cup \pi_{\sigma_0}(x,z_J)\]
    Lastly, let $\pi_\sigma(x,y_I)$ be the partial type $\exists w\exists z_J\chi_{\sigma}(x,y_I,w,z_J)$ (by which we mean the partial type obtained by existentially quantifying $w,z$ in every finite conjunction of formulas from $\chi_\sigma$; in a saturated model this has precisely the meaning suggested by the abuse of notation).
    
    We verify that $\pi_\sigma$ is as desired. Suppose that $(\mathbb{M}^+)^{eq}\models \pi_\sigma(a,b_I)$. Then there are $c,c_J$ such that $(\mathbb{M}^+)^{eq}\models\chi_\sigma(a,b_I,c,c_J)$. Letting $B$ denote the set enumerated by $b_I$ and $C$ that by $c_J$, this implies the following: $(\mathbb{M}^+)^{eq}\models\varphi(a,c)$, $\varphi(x,c)$ $k$-divides over $B$, $C$ contains $\mathrm{acl}^+(Bc)\cap \mathbb{}M^{eq}$, and $\tp(a/C)$ has a $T^+$-robust tree of shape $\sigma_0$ realized by $a$. Call this latter tree $\tau_0$. Then if we prepend $\tp(a/B)$ to $\tau_0$ as a root, we obtain a $T^+$-robust tree of shape $\sigma$ for $\tp(a/B)$, realized by $a$.

    Conversely, suppose that $b_I$ enumerates a set $B$ and that $\tp(a/B)$ has a $T^+$-robust tree of shape $\sigma$ realized by $a$; call this tree $\tau$. Let $p:=\tp(a/B)$ and let $q:=\tp(a/C)$ denote the unique child of $p$ in $\tau$, and let $(\varphi(x,y),k)$ be the value of $\sigma$ corresponding to this node; that is, $(\varphi(x,y),k)=\sigma(r^+)$. Then $C$ contains a parameter $c$ such that $\varphi(x,c)\in q$, $\varphi(x,c)$ $k$-divides over $B$, and $C\supseteq \mathrm{acl}^+(Bc)\cap\mathbb{M}^{eq}$. Accordingly, it is possible to enumerate a subset $C'\subseteq C$ as $c_J$ in such a way that $(\mathbb{M}^+)^{eq}\models\rho(b_Ic,c_J)$. We also have $(\mathbb{M}^+)^{eq}\models\varphi(a,c)$ and $(\mathbb{M}^+)^{eq}\models\zeta(c,b_I)$. Restricting $\tau,\sigma$ to the tree above $r^+$, we get trees $\tau_0$ and $\sigma_0$ of rank $\alpha$ with $\tau_0$ being a $T^+$-robust tree for $q$ of shape $\sigma_0$ realized by $a$. If we replace $q$ in $\tau_0$ by $q':=q\mid C'$ then we still have a $T^+$-robust tree of shape $\sigma_0$ for $q'$ realized by $a$ and hence $(\mathbb{M}^+)^{eq}\models\pi_{\sigma_0}(a,c_J)$ by inductive hypothesis. So altogether we have that $(\mathbb{M}^+)^{eq}\models\chi_\sigma(a,b_I,c,c_J)$ and so $(\mathbb{M}^+)^{eq}\models\pi_\sigma(a,b_I)$.\\
    \\
    Case $\alpha$ is a limit: For each $\beta<\alpha$ recall the definition of the distinguished subtree of $\mathcal{T}_\alpha$ with rank $\beta$ given previously, and of the restriction $\sigma_\beta$.

    We define $\pi_\sigma(x,y_I):=\bigcup_{\beta<\alpha}\pi_{\sigma_\beta}(x,y_I)$. Then $\pi_\sigma$ is as desired if we can observe that $p:=\tp(a/b_I)$ has a $T^+$-robust tree of shape $\sigma$ realized by $a$ iff for all $\beta<\alpha$, $p$ has a $T^+$-robust tree of shape $\sigma_\beta$ realized by $a$. The forward implication is clear, and for the reverse implication, taking the union of all the hypothesized trees $\tau_\beta$ gives exactly a $T^+$-robust tree of shape $\sigma$ for $p$ realized by $a$.
\end{proof}

The remaining arguments will require further hypotheses on the shape of a forking tree:

\begin{definition}
    Let $\sigma$ be a function with domain $\mathcal{T}_{\lambda+\omega}$ where $\lambda$ is a limit ordinal (or 0). Then we say $\sigma$ is \textit{chain-like} if:
    \begin{enumerate}
        \item The tree above any node of rank $\lambda$ is the same 
        \item The tree given by (1) is, except for possibly the root node, identical to the distinguished subtree $\sigma_\lambda$
        \item Every node of rank $\lambda+n$ is identical, for fixed $n$ (note the case $n=0$ already follows from condition 1).
    \end{enumerate}
\end{definition}
To justify the above terminology, note that a chain-like tree is exactly the sort of forking tree that arises from a certain chain of forking extensions: let $p$ be a type of rank $\lambda+\omega$ and let $q$ be some extension of rank $\lambda$. Suppose there is an infinite sequence of intermediate forking extensions $q=q_0\supseteq q_1\supseteq q_2\supseteq\dots\supseteq p$. Take some forking tree of rank $\lambda$ above $q_0$. It is clear that these data, the tree and the chain, together are enough to specify a forking tree of rank $\lambda+\omega$ over $p$ and that this tree will be chain-like in the above sense. It is also clear that a chain-like forking tree has a chain-like shape.

\begin{definition}
    We say that a function on $\mathcal{T}_\alpha$ is \textit{almost uniform} if the subtree above any node of rank $\lambda+\omega$, where $\lambda$ is a limit or zero, is chain-like.
\end{definition}

Since taking complete subtrees preserves the rank of nodes, it is clear that any subtree above any node of an almost uniform tree is almost uniform. Likewise, on an almost uniform tree of limit rank, any distinguished subtree is almost uniform.

\begin{remark}
    If we wish to make the connection with chain patterns from before, we could define a \textit{uniform} function on $\mathcal{T}_\alpha$ to be one which assigns the same value to every node of a fixed rank. Clearly the data of a uniform tree of formulas is essentially a chain pattern. A chain pattern of rank $\alpha$ gives rise to a forking tree of uniform shape, and conversely a straightforward compactness argument shows that if a type has a forking tree of uniform shape, then it has a chain pattern specified by the same sequence of formulas and dividing numbers that determine that uniform shape.

    Hence we see that having a forking tree of almost-uniform shape of some rank is a weakening of our previous notion of having a chain pattern of that rank.
\end{remark}

The point of the next few propositions is that if a type $p$ has a tree of almost-uniform shape $\sigma$, then almost uniformity will allow us to run a compactness argument to show that $p$ has a $T^+$-robust tree of shape $\sigma$.

\begin{prop}\label{AlmostUnifProp}
    If $\sigma$ is an almost uniform shape of limit rank $\alpha$ and $y_I$ some sequence of variables, then there is an increasing sequence of ordinals $\beta_i<\alpha$ (indexed by some ordinal $\delta$) cofinal in $\alpha$ such that, for each $i$ and each $\gamma<\beta_i$, $\pi_{\sigma_{\beta_i}}(x,y_I)$ entails $\pi_{\sigma_\gamma}(x,y_I)$ where these are the partial types defined in Proposition \ref{TypeDefProp} 
\end{prop}
\begin{proof}
    Either $\alpha$ has a cofinal subsequence of limit ordinals or it has the form $\lambda+\omega$ for $\lambda$ a limit or zero. In the former case, let $\beta_i$ be such a subsequence. Then it is immediate from the definition of $\pi_\sigma$ at limit ranks that $\pi_{\sigma_{\beta_i}}(x,y_I)$ entails $\pi_{(\sigma_{\beta_i})_\gamma}(x,y_I)$ for $\gamma<\beta_i$. Recall also that $(\sigma_{\beta_i})_{\gamma}=\sigma_\gamma$, hence $\pi_{\sigma_{\beta_i}}(x,y_I)$ entails $\pi_{\sigma_{\gamma}}(x,y_I)$.

    In the second case, say $\alpha=\lambda+\omega$, and let $\beta_i:=\lambda+i$. Let $\gamma<\lambda=\beta_0$. We first observe that $\pi_{\sigma_\lambda}(x,y_I)$ entails $\pi_{\sigma_\gamma}(x,y_I)$. If $\lambda=0$ the condition is vacuous; otherwise it follows from the fact that $\sigma_\gamma=(\sigma_\lambda)_\gamma$ since $\lambda$ is a limit ordinal. For the rest, it will then suffice to show that $\pi_{\sigma_{\lambda+i+1}}(x,y_I)$ entails $\pi_{\sigma_{\lambda+i}}(x,y_I)$. Indeed, deleting the unique node of rank $\lambda+i$  from a $T^+$-robust tree of shape $\sigma_{\lambda+i+1}$ yields a $T^+$-robust tree of shape $\sigma_{\lambda+i}$: if $i=0$ this follows from condition (2) in the definition of chain-like; if $i>0$ then this follows from conditions (1) and (3) of chain-like.
\end{proof}

\begin{prop}\label{StarProp2}
    Suppose that $p\in S_T(A)$ has a forking tree (in the sense of $T$) of almost uniform shape $\sigma$. Then $p$ also has  a $T^+$-robust tree of shape $\sigma$.
\end{prop}
\begin{proof}
    We proceed by induction on $\alpha$, the rank of $\sigma$. The case $\alpha=0$ is clear as all the extra conditions for a $T^+$-robust tree are vacuous.\\
    \\
    \noindent Case $\alpha+1$: let $r,r^+$ denote the root of $\mathcal{T}_{\alpha+1}$ and its unique child, respectively. Let $\tau$ be a rank $\alpha+1$ tree for $p$ of shape $\sigma$. Let $\sigma_0,\tau_0$ be the restrictions of $\sigma,\tau$ to the tree above $r^+$. Let $q=\tau(r^+)$, the unique child of $p$ in $\tau$. Then $\tau_0$ is a rank $\alpha$ forking tree of shape $\sigma_0$ for $q$. Let $q'$ denote a nonforking extension of $q$ to $B':=\mathrm{acl}^+(B)\cap \mathbb{M}^{eq}$. By Proposition \ref{TreeExtProp}, $q'$ also has a rank $\alpha$ tree of shape $\sigma_0$. By inductive hypothesis, $q'$ has a $T^+$-robust tree of shape $\sigma_0$; call it $\tau_0'$. If $(\varphi(x,y),k)$ denotes $\sigma(r^+)$, note that $q'$ contains an instance $\varphi(x,b)$ $k$-dividing over $A$ and that $\mathrm{acl}^+(Ab)\subseteq B'$. Thus, prepending $p$ to $\tau_0'$ as root node gives a $T^+$-robust tree for $p$ of shape $\sigma$.
    
    Now suppose $\alpha$ is a limit ordinal. We need to show consistency of $p(x,A)\cup \pi_\sigma(x,a_I)$ where $a_I$ enumerates $A$. By compactness and the definition of $\pi_\sigma$ at limit ranks, it is enough to show consistency of $p(x,A)\cup \bigcup_{\gamma<\beta}\pi_{\sigma_\gamma}(x,a_I)$ for arbitrary $\beta<\alpha$. By Proposition \ref{AlmostUnifProp} there is an increasing sequence $(\beta_i)_{i<\delta}$ of ordinals cofinal in $\alpha$ such that for each $i$ and each $\gamma<\beta_i$, $\pi_{\sigma_{\beta_i}}(x,y_I)$ entails $\pi_\gamma(x,y_I)$. Therefore, it will suffice to show the consistency of $p(x,A)\cup \pi_{\sigma_{\beta_{i}}}(x,y_I)$ for arbitrarily large $i$. But $p$ has a forking tree of shape $\sigma$, hence a forking tree of shape $\sigma_{\beta_i}$, hence a $T^+$-robust tree of shape $\sigma_{\beta_i}$ by the inductive hypothesis, so this partial type is consistent.
\end{proof}

We next verify that every type really does have a tree of almost-uniform shape. Note the similarity of the proof with that of Proposition \ref{IntermediateProp}. Indeed, it would be possible to define a parametrized notion of $\omega^\alpha$-almost-uniform for trees (where we view our current definition as being really ``$\omega$-almost-uniform"), and a similar argument to what follows would show the existence of such trees for countable $\alpha$, but there is no need for any more than the uniformity we have already defined.

\begin{prop}\label{AlmostUnifExistenceProp}
    Let $p$ be a type of rank $\geq\alpha$ in a supersimple theory. Then $p$ has a forking tree of rank $\alpha$ with almost uniform shape. 
\end{prop}
\begin{proof}
    By induction on $\alpha$. The only interesting steps are ranks $\lambda+\omega$ where $\lambda$ is a limit or zero. Let us fix some conventions for describing shapes of rank $\lambda+n$. In the case $n=0$ we may simply specify some shape $\sigma$ of rank $\lambda$. To specify a shape of rank $\lambda+n$ for $n\geq 1$, we will write $\langle \sigma, (\varphi_1,k_1),\dots,(\varphi_{n},k_{n})\rangle$ to indicate the shape of rank $\lambda+n$ where the tree above the rank $\lambda$ node has shape $\sigma$, and the node at rank $\lambda+i$ is $(\varphi_i,k_i)$. In order to show that $p$ with SU-rank $\lambda+\omega$ has a rank $\lambda+\omega$ tree of almost uniform shape, it will suffice to show that there is an almost-uniform shape $\sigma$ of rank $\lambda$ and a sequence of dividing data $(\varphi_i,k_i)$ such that for all $n<\omega$, $p$ has a tree of shape $\langle \sigma, (\varphi_1,k_1),\dots,(\varphi_n,k_n)\rangle$. Recall that when checking whether a forking tree is of some specified shape, the root node of the shape is essentially irrelevant, so we will sometimes write ``$p$ has a tree of shape $\langle\sigma, (\varphi_1,k_1),\dots,(\varphi_n,k_n),*\rangle$" to mean that for any choice of dividing data $(\varphi,k)$, $p$ has a tree of shape $\langle\sigma, (\varphi_1,k_1),\dots,(\varphi_n,k_n),(\varphi,k)\rangle$.

    Now we generate the claimed sequence of dividing data by induction. Let $p$ be a type of rank $\lambda+\omega$. After taking a nonforking extension, we may assume $p$ is over a model $M$. $p$ has a forking extension $q$ of rank $\lambda$. After taking a nonforking extension, we may assume $q$ is over an $|M|^+$-saturated model $N$. Apply Lemma \ref{EvenBetterBuechler} to obtain an intermediate forking extension $r_1$ of rank $\lambda+1$. The forking of $q$ over $r_1$ is witnessed by some dividing data $(\varphi_0,k_0)$. \textit{A fortiori} the same data witness the dividing of $q$ over $p$. By the inductive hypothesis, $q$ has a rank $\lambda$ forking tree of almost-uniform shape $\sigma$. Replace the root node of $\sigma$ by $(\varphi_0,k_0)$. At this point, we can see that $p$ and $r_1$ each have a forking tree of shape $\langle \sigma, *\rangle$. Note also that $\mathrm{SU}(r)$ is strictly between $\lambda$ and $\lambda+\omega$. Replace $r_1$ by a nonforking extension to an $|M|^+$-saturated model. By Proposition \ref{TreeExtProp}, $r_1$ still has a tree of shape $\langle \sigma, *\rangle$. Again take a forking extension $r_2$ intermediate to $p$ and $r_1$. Let $(\varphi_1,k_1)$ witness dividing of $r_1$ over $r_2$. Since $r_1$ has a tree of shape $\langle \sigma,*\rangle$, in particular that tree is of shape $\langle \sigma, (\varphi_1,k_1)\rangle$, and we now see that $p,r_2$ both have forking trees of shape $\langle \sigma, (\varphi_1,k_1),*\rangle$.

    Suppose by induction that we have determined $(\varphi_1,k_1),\dots,(\varphi_n,k_n)$ such that $p$ has a forking tree of shape $\langle \sigma, (\varphi_1,k_1),\dots,(\varphi_n,k_n),*\rangle$ as well as a forking extension $r_{n+1}$ which has a tree of shape $\langle \sigma, (\varphi_1,k_1),\dots,(\varphi_n,k_n),*\rangle$. After taking a nonforking extension, we may suppose $r_{n+1}$ is over an $|M|^+$-saturated model. Note that $\lambda+n<\mathrm{SU}(r_{n+1})<\lambda+\omega$. Therefore we can find an intermediate forking extension $r_{n+2}$ of $p$ and data $(\varphi_{n+1},k_{n+1})$ which witness the dividing of $r_{n+1}$ over $r_{n+2}$, hence also over $p$. Again we see that $r_{n+2},p$ each have forking trees of shape $\langle \sigma, (\varphi_{1},k_{1}),\dots,(\varphi_{n+1},k_{n+1}),*\rangle$, completing the induction.
\end{proof}

And now we are able to conclude our main result:
\begin{theorem}\label{FullThm}
    Let $p\in S_T(A)$ be a complete $T$-type of rank $\mathrm{SU}(p)\geq\alpha$. Then there exists a complete $T^+$-type $p^+\in S_{T^+}(A)$ extending $p$ with rank $\mathrm{SU}^+(p^+)\geq \alpha$.
\end{theorem}
\begin{proof}
    Replace $p$ by a nonforking extension to $\mathrm{acl}^+(A)\cap\mathbb{M}^{eq}$. By Proposition \ref{AlmostUnifExistenceProp}, $p$ has (in $T$) a forking tree of rank $\alpha$ with almost-uniform shape $\sigma$. By Proposition \ref{StarProp2}, $p$ has a $T^+$-robust tree of shape $\sigma$. Then by Propositions \ref{PruneProp} and \ref{GoalProp}, there exists some $p^+$ as claimed.
\end{proof}

\section{References}
\noindent[A] Adler, Hans. Explanation of Independence. arXiv preprint math/0511616 (2005).
\newline
\newline
\noindent[B] Buechler, Steven. Maximal Chains in the Fundamental Order. \textit{The Journal of Symbolic Logic}, vol. 51, no. 2, 1986, pp. 323–26.
\newline
\newline
\noindent[N] N\"ubling, Herwig. Reducts of stable, CM-trivial theories. \textit{The Journal of Symbolic Logic}, vol. 70, no. 4, 2005, pp. 1025–36.
\newline
\newline
\noindent[W] Wagner, Frank. \textit{Simple Theories}. Mathematics and its applications, vol. 503.
Kluwer Academic Publishers, Dordrecht, Boston, and London, 2000.

\end{document}